\documentclass[11pt]{amsart}
\usepackage{amssymb, amsmath, amscd, mathrsfs, graphicx}

\makeatletter
\def\@tocline#1#2#3#4#5#6#7{\relax
  \ifnum #1>\c@tocdepth 
  \else
    \par \addpenalty\@secpenalty\addvspace{#2}%
    \begingroup \hyphenpenalty\@M
    \@ifempty{#4}{%
      \@tempdima\csname r@tocindent\number#1\endcsname\relax
    }{%
      \@tempdima#4\relax
    }%
    \ifnum#1=1\large\fi
    \ifnum#1=2\small\fi
    \parindent\z@ \leftskip#3\relax \advance\leftskip\@tempdima\relax
    \rightskip\@pnumwidth plus4em \parfillskip-\@pnumwidth
    #5\leavevmode\hskip-\@tempdima #6\nobreak\relax
    \hfil\hbox to\@pnumwidth{\@tocpagenum{#7}}\par
    \nobreak
    \endgroup
  \fi}
\makeatother

\newlength{\Width}
\newlength{\Height}  
\newlength{\Depth}

\newtheorem{thm}{Theorem}[section]
\newtheorem{lem}[thm]{Lemma}

\newtheorem{prop}[thm]{Proposition}

\newtheorem{periodthm}{Theorem}[section]

\theoremstyle{definition}

\newtheorem{rem}[thm]{Remark}

\theoremstyle{remark}

\numberwithin{equation}{section}

\def\Z{{\mathbb Z}}
\def\C{{\mathbb C}}
\def\Q{{\mathbb Q}}

\allowdisplaybreaks

\title[The period matrix of the hyperelliptic curve $w^2=z^{2g+1}-1$]
{The period matrix of the hyperelliptic curve $w^2=z^{2g+1}-1$}

\author[Yuuki Tadokoro]{Yuuki Tadokoro}
\address{Natural Science Education,
Kisarazu National College of Technology, 2-11-1 Kiyomidai-Higashi,
Kisarazu, Chiba 292-0041, Japan}
\email{tado\char`\@nebula.n.kisarazu.ac.jp}

\begin{document}

\maketitle

\begin{abstract}
A geometric algorithm is introduced for finding
a symplectic basis of the first integral homology group
of a compact Riemann surface, which is
a $p$-cyclic covering of $\C P^1$ branched over 3 points.
The algorithm yields a previously unknown symplectic basis of
the hyperelliptic curve defined by the affine equation $w^2=z^{2g+1}-1$
for genus $g\geq 2$.
We then explicitly obtain the period matrix of this curve,
its entries being elements of the $(2g+1)$-st cyclotomic field.
In the proof, the details of our algorithm play no significant role.
\end{abstract}

\section{Introduction}
Let $X$ be a compact Riemann surface of genus $g\geq 2$ or
smooth projective algebraic curve over $\C$.
The period matrix $\tau_g$ of $X$
depends only on the choice of symplectic basis
of the first integral homology group $H_1(X;\Z)$.
It is known that $\tau_g$ is symmetric and its imaginary part
is positive definite.
The Jacobian variety $J(X)$ of $X$ is defined by a complex torus
$\C^g/{(\Z^g+ \tau_g \Z^g)}$.
Torelli's theorem states that two given Riemann surfaces
$X$ and $Y$ are biholomorphic if and only if $J(X)$ and
$J(Y)$ are isomorphic as polarized abelian varieties.
It implies that $\tau_g$ determines the complex structure of $X$.
In general, calculating $\tau_g$ is not easy; 
the difficulty is in finding a symplectic basis of $H_1(X; \Z)$.
Tretkoff and Tretkoff \cite{0557.30036}
gave a method to compute $\tau_g$, using {\it Hurwitz systems}.
Andersen, Bene, and Penner \cite{1196.57013} showed a way of
finding a symplectic basis of $H_1(X; \Z)$ using {\it chord slides}
for linear chord diagrams.
By combining these two methods,
we explicitly write down a geometric algorithm for
finding a symplectic basis of the first integral homology groups
of $p$-cyclic coverings of $\C P^1$ branched over 3 points
for prime number $p\geq 5$.
We call it the {\it chord slide algorithm}.
The Frobenius method that Tretkoff and Tretkoff used is steady but
not easy to apply to compact Riemann surfaces of generic genus $g$.
We can apply the chord slide algorithm to special surfaces
of generic genus $g$ with good linear chord diagrams.
Furthermore, 
we illustrate calculating the period matrix
for this kind of Riemann surfaces.
We compute the period matrices of the hyperelliptic
and Klein quartic curves defined by the affine equations
$y^{7}=x(1-x)$ and $y^{7}=x(1-x)^2$, respectively.
These computations are already known;
see \cite{0557.30036} for the first curve and \cite{1200.14064,1038.14011,0219.30007, 0911.14021, schindler1991jacobische,1222.14058,1073.14529} for the second.

For any {\it odd} number $q\geq 5$,
let $C_{q,1}$ be the smooth projective curve over $\C$
defined by the affine equation $y^{q}=x(1-x)$.
This is biholomorphic to the hyperelliptic curve
defined by the affine equation $w^2=z^{2g+1}-1$ of genus $g=(q-1)/2$.
We can apply the chord slide algorithm to this curve $C_{q,1}$
and obtain a symplectic basis 
$\{a_i,b_i\}_{i=1,2,\ldots,g}$ of $H_1(C_{q,1}; \Z)$
different to the well-known basis in \cite{0192.58201}.
The advantage of our method is its applicability to
other curves, for example, nonhyperelliptic ones.

For generic genus, few examples of period matrices are known.
Schindler \cite{0801.14008} computed the period matrices of three types 
of hyperelliptic curves of genus $g\geq 2$.
These matrices are the only examples as far as we know.
Kamata \cite{1016.32010} introduced an algorithm for calculating those
of Fermat-type curves.
We know of no other algorithms except Kamata's and Tretkoff and Tretkoff's.
For explicit computations for cases of low genus, see \cite{0782.14026,0962.14022,0856.30031}
except for the above hyperelliptic and Klein quartic curves.
We remark that Streit \cite{1001.14011} studied the period matrices
from the viewpoint of representation theory.
Tashiro, Yamazaki, Ito, and Higuchi \cite{0948.14501}
computed the periods on $C_{q,1}$.
We obtain the period matrix $\tau_g$
of $C_{q,1}$ using the inverse of the Vandermonde matrix.
Our original contribution to the computation of $\tau_g$
is in finding the inverse matrix $\Omega_A^{-1}$, which is defined in Section \ref{Period matrices for hyperelliptic curves}.
Schindler \cite{0801.14008} obtained one for the same hyperelliptic curve
of genus $g\geq 2$ defined by the affine equation $w_1^2=z_1(z_1^{2g+1}-1)$,
which is biholomorphic to $C_{q,1}$.
This result contains a recurrence relation.
However, we have an explicit representation of $\tau_g$.
Set $\zeta=\zeta_{q}=\exp(2\pi\sqrt{-1}/{q})$.
A symplectic basis $\{A_i, B_i\}_{i=1,2,\ldots,g}$ of $H_{1}(C_{q,1};\Z)$
is defined later.
For variables $x_1,x_2,\ldots, x_n$,
we denote by $\sigma_{i}(x_1,x_2,\ldots, x_n)$ the symmetric polynomial
\[
\sum_{1\leq j_1<\cdots < j_{i}\leq n} 
x_{j_1}\cdots x_{j_i}\]
for $1\leq i\leq n$ and $\sigma_{0}(x_1,x_2,\ldots, x_n)=1$.
\begin{periodthm}
With respect
 to the symplectic basis $\{A_i, B_i\}_{i=1,2,\ldots,g}$, the period matrix $\tau_g$ of $C_{q,1}$ is
\[
\tau_g=
\left(
\sum_{k=1}^{g}
\dfrac{(-1)^{i+g}}{2g+1}(1-\zeta^{2kj})
\sigma_{g-i}(\zeta^2,\zeta^4,
\ldots,\widehat{\zeta^{2j}},\ldots,\zeta^{2g})
\prod_{m=g-k+1}^{2g-k}(1-\zeta^{2m})
\right)_{i,j},
\]
where the `hat' symbol \,$\widehat{\phantom{a}}$ over an element $\zeta^{2j}$
indicates that this element is deleted from the sequence $\zeta^2,\zeta^4,\ldots,\zeta^{2g}$.
\end{periodthm}

\tableofcontents

\section{Algorithm}
\label{Algorithm}
We begin by writing down the geometric algorithm,
called the chord slide algorithm,
for finding a symplectic basis of the first integral homology groups
of the smooth projective algebraic curves defined by the affine equation
$y^p=x^l(1-x)^m$.
Here $p\geq 5$ is a prime number, $l, m$ are coprime, and
$1\leq l,m, l+m<p-1$.
We denote this curve by $X_{p,l,m}$.
This is a Riemann surface of genus $g=(p-1)/{2}$ and
can be considered as a $p$-sheeted cyclic coverings of $\C P^1$
branched over $\{0,1,\infty\}\subset \C P^1$.
In particular, we simply write the curve $X_{p,1,m}=C_{p,m}$ for the case $l=1$.
Throughout this section,
we work with the Klein quartic $C_{7,2}$ as an example.
Moreover, 
we detail an algorithm for calculating the period matrix
of $X_{p,l,m}$ using
holomorphic 1-forms of Bennama and Carbonne \cite{0842.14022}.
It is known that up to isomorphism there are only two curves $X_{7,l,m}$.
They are $C_{7,1}$ and $C_{7,2}$; see \cite[\S 1.3.2]{pre06000861} for example.
We next calculate their period matrices.

\subsection{Dessins d'enfants}
\label{Dessins d'enfants}
Let $X$ be a smooth projective algebraic curve over a field $k$.
We assume that $k$ is $\C$ and there exists a covering $\pi\colon X\to \C P^1$
branched over $\{0,1,\infty\}\subset \C P^1$.
The inverse image $\pi^{-1}([0,1])$ in $X$ of the unit interval in
$\C P^1$ is called a dessin d'enfants \cite{0901.14001}.
It is a topological bipartite graph illustrated on the Riemann surface $X$.
Belyi \cite{0429.12004} proved that all
algebraic curves over $\overline{\Q}$ correspond to dessins d'enfants.
The map $\pi$ is often called the Belyi map; see also \cite{1042.14016}.
In the rest of this paper, we assume 
that $X$ is $X_{p,l,m}$
and $\pi\colon X\ni (x,y)\mapsto x\in \C P^1$
is a $p$-cyclic covering
branched over $\{0,1,\infty\}\subset \C P^1$.
Set the order $p$ holomorphic automorphism $\sigma(x,y)=(x,\zeta_p y)$.
Here, we denote $\zeta_p=\exp(2\pi\sqrt{-1}/{p})$.
Let $y_0(t)$ be a real analytic function $\sqrt[p]{t^l(1-t)^m}$.
A continuous path $I_0\colon [0,1]\to X$ is defined by
the equation $I_0(t)=(t, y_0(t))\in X$ for $0\leq t\leq 1$.
We immediately obtain
$\pi(I_0)=[0,1]\subset \C P^1$ and
the dessin d'enfants
$\pi^{-1}([0,1])=\cup_{i=0}^{p-1} \sigma_{\ast}^{i}(I_0)$.
We call $\pi^{-1}(0)$ and $\pi^{-1}(1)$
the white and black vertices, respectively.
The dessin $\pi^{-1}([0,1])$ in $X$ is a bipartite graph.
Take a point $b_i$ on $\sigma_{\ast}^{i}(I_0)$ except for endpoints for each $i$.
For the Klein quartic $C_{7,2}$,
we draw a dessin d'enfants $\pi^{-1}([0,1])$; see Figure \ref{dessins-denfants}.

\begin{figure}[htbp]
\begin{center}
 \input{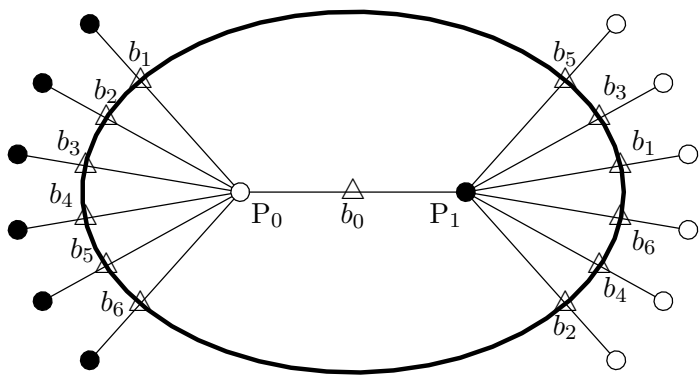}
\caption{Dessin d'enfants for the Klein quartic $C_{7,2}$
\label{dessins-denfants}}
\end{center}
\end{figure}

\subsection{Intersection numbers}
\label{Intersection numbers}
We now introduce the method of Tretkoff and Tretkoff \cite{0557.30036},
based on the Hurwitz system,
from dessins d'enfants to the intersection numbers of the loops in $X$.
Let $V$ be the set of $2g=p-1$ labeled points on the unit circle $S^1$.
A chord diagram on $V$ is a set of oriented simple chords between points of $V$.

For $i=1,2,\ldots,p-1$, let $c_i$ denote the loop
$I_0 \cdot \sigma_{\ast}^{i}(I_0)^{-1}$ in $X$.
Here, the product $I_0 \cdot \sigma_{\ast}^{i}(I_0)^{-1}$
signifies that we traverse $I_0$ first and then $\sigma_{\ast}^{i}(I_0)^{-1}$.
It follows that the dessin d'enfants $\pi^{-1}([0,1])$
equals the union $\cup_{i=0}^{p-1}c_i$.
We deform the dessin topologically
and consider $c_i$ as not loops but chords.
We get a chord diagram
and compute the intersection numbers $c_i\cdot c_j=0$ or $\pm 1$.
This satisfies the property in \cite[\S 8.1]{1214.57019}.
Let $(a_{i,j})$ denote the matrix with $(i,j)$-th entry $a_{i,j}$.
If the $2g\times 2g$ intersection matrix $(c_i\cdot c_j)$ is regular,
then $\{c_i\}_{i=1,2,\ldots,2g}$ is
a basis of the first integral homology group $H_1(X; \Z)$.

For $C_{7,2}$,
we obtain the chord diagram in Figure \ref{chord-digram}
that corresponds to Figure \ref{dessins-denfants}.
For convenience, the origin and terminal point of $c_i$
is denoted by $i$ and $\bar{i}$, respectively.
The intersection matrix is as follows \cite[pp. 482]{0557.30036}
\[
\left(
\begin{array}{cccccc}
0 &0 &1 &0 &1 &0 \\
0 &0 &1 &1 & 1&1 \\
-1 &-1 &0 &0 &1 &0 \\
0 &-1 &0 &0 &1 &1 \\
-1 &-1 &-1 &-1 &0 &0 \\
0 &-1 &0 &-1 &0 &0 \\
\end{array}
\right).
\]
We can choose from many bases of $H_1(X; \Z)$.
For example, we \cite{1222.14058} can choose the basis
$\{\ell_i\}_{i=1,2,\ldots,6}$ such that
$\ell_i=c_{i-1}\cdot c_{i}^{-1}$ for $C_{7,2}$.

\begin{figure}[htbp]
\begin{center}
 \input{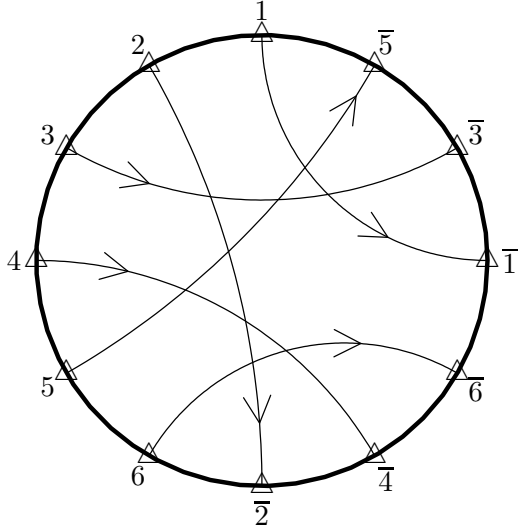}
\caption{Chord diagram of $C_{7,2}$
\label{chord-digram}}
\end{center}
\end{figure}

We write down the intersection number $c_i\cdot c_j$ for $X_{p,l,m}$.
For $1\leq i\leq p-1$, the integer $i_{l}\in \{1,2,\ldots,p-1\}$
is uniquely determined such that $i_{l}l\equiv i$ modulo $p$.
We define $i_m$ similarly.
Draw the loops $c_i$ and $c_j$.
The initial point $i$ of $c_i$ is
the $i_l$-th point moving counter-clockwise from the point $l$ and
the terminal point $\overline{i}$ is the $i_m$-th point from the point $\overline{m}$.
We give the associated chord diagram in Figure \ref{chord-digram_2}
for the case $j_l-i_l>0$ and $j_m-i_m>0$.
In this case, we have $c_i\cdot c_j=1$.
\begin{figure}[htbp]
\begin{center}
 \input{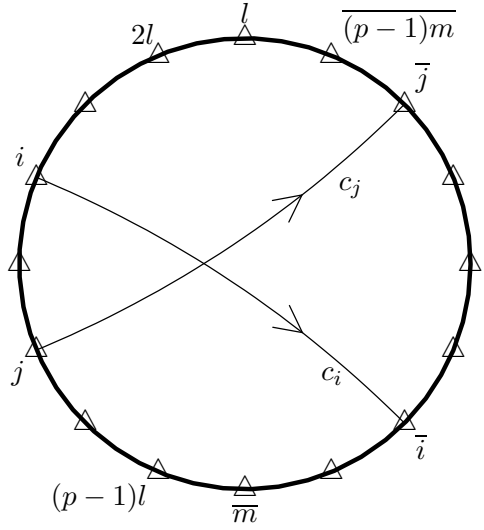}
\caption{$c_i\cdot c_j=1$ for $j_l-i_l>0$ and
$j_m-i_m>0$
\label{chord-digram_2}}
\end{center}
\end{figure}

\begin{rem}
\label{intersection numbers}
We have the intersection number
\[
c_i\cdot c_j
=\left\{
  \begin{array}{rl}
   1 & (j_l-i_l>0 \text{ and }
j_m-i_m>0),\\
  -1 & (j_l-i_l<0 \text{ and }
j_m-i_m<0),\\
   0 & (\text{otherwise}).
  \end{array}
 \right.
\]
\end{rem}
This is easily seen to be true.

\subsection{Linear chord diagrams}
We identify chord diagrams with linear chord diagrams.
A linear chord diagram in the plane with $k$
chords is defined as an interval $[0, 2k]$,
together with $k$ oriented simple arcs in the upper half plane
between the integer points $\{1,2,\ldots, 2k\}$.

Cut open a chord diagram from Section \ref{Intersection numbers}
at a certain point on $S^1$.
By identifying the $2g$ chords with loops $\{c_1,c_2,\ldots,c_{2g}\}$,
we have the corresponding linear chord diagram. 
The end points of $c_i$'s on $S^1$ determine the intersection number $c_i \cdot c_j$.
We remark that $c_i \cdot c_j$ is independent of the choice of cutting points,
but the chord slide method depends on it.

For $C_{7,2}$, choosing point $3$ in Figure \ref{chord-digram} 
produces the linear chord diagram in Figure \ref{linear-chord-diagram}.

\begin{figure}[htbp]
\begin{center}
 \input{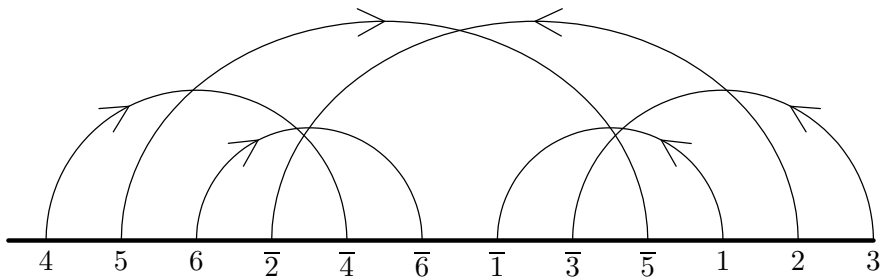}
\caption{A linear chord diagram of $C_{7,2}$
\label{linear-chord-diagram}}
\end{center}
\end{figure}

\subsection{Chord slides}
Andersen, Bene, and Penner \cite{1196.57013}
used chord slides
for the linear chord diagrams.
They studied chord slides with the Whitehead moves
on the dual of the fat graphs embedded in a surface of genus $g$
with one boundary component; see also \cite{1214.57019}.
We simply use chord slides to compute the
intersection numbers and find the matrix $T$ such that
$(c_1,c_2,\ldots,c_{2g}){}^tT$ is a symplectic matrix.

For the linear chord diagram,
we define a {\it chord slide of $c_i$ along $c_j$ for the same position}
by the transformation from $(c_1,c_2,\ldots,c_{2g})$
to $(c^{\prime}_1,c^{\prime}_2,\ldots,c^{\prime}_{2g})$ such that
\[
c_k^{\prime}=
\left\{
 \begin{array}{ll}
  c_i-c_j & (k=i),\\
  c_k      & (k\neq i).
 \end{array}
\right.
\]
as homology classes.
For the {\it opposite position},
the $c_i^{\prime}$ is replaced with $c_i+c_j$.
We define the {\it position} of a chord slide of $c_i$ along $c_j$
in the linear chord diagram.
It is the {\it same position} that
$c_i$ initially (ultimately) moves in the direction towards 
the origin (terminal) point of $c_j$.
The others are {\it opposite positions}.
We remark that the position does not depend on
the orientation of the chord; see Figures \ref{linear-chord-diagram_proof_1}
and \ref{opposite_direction}.
Using the origin and terminal points,
we simply write the chord slide in Figure \ref{linear-chord-diagram_proof_1}
and \ref{opposite_direction} as $c\to d$ and $\overline{c}\to d$, respectively.
\begin{figure}[htbp]
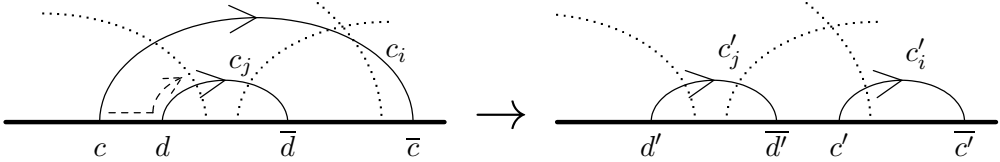

\begin{center}
 \input{linear-chord-diagram_proof_1.tex}\quad
 {\huge$\rightarrow$}\quad
 \input{linear-chord-diagram_proof_2.tex}
\caption{A chord slide of $c_i$ along $c_j$ for the same position
\label{linear-chord-diagram_proof_1}}
\end{center}
\end{figure}
\begin{figure}[htbp]
\begin{center}
 {\unitlength=1cm%
\begin{picture}%
(7.00,1.10)(-0.50,-0.10)%
\special{pn 8}%
\special{pa 1181 0}\special{pa 1180 -16}\special{pa 1178 -34}\special{pa 1174 -50}%
\special{pa 1169 -65}\special{pa 1161 -82}\special{pa 1153 -97}\special{pa 1142 -114}%
\special{pa 1130 -129}\special{pa 1118 -143}\special{pa 1103 -157}\special{pa 1088 -170}%
\special{pa 1071 -182}\special{pa 1053 -194}\special{pa 1034 -204}\special{pa 1013 -215}%
\special{pa 992 -224}\special{pa 971 -232}\special{pa 947 -240}\special{pa 924 -246}%
\special{pa 899 -252}\special{pa 874 -256}\special{pa 850 -259}\special{pa 824 -261}%
\special{pa 800 -262}\special{pa 776 -262}\special{pa 750 -261}\special{pa 725 -259}%
\special{pa 700 -256}\special{pa 674 -251}\special{pa 652 -246}\special{pa 627 -240}%
\special{pa 605 -233}\special{pa 584 -225}\special{pa 562 -215}\special{pa 543 -206}%
\special{pa 523 -194}\special{pa 505 -183}\special{pa 489 -171}\special{pa 473 -158}%
\special{pa 458 -144}\special{pa 445 -129}\special{pa 433 -114}\special{pa 423 -99}%
\special{pa 414 -83}\special{pa 407 -67}\special{pa 401 -51}\special{pa 397 -34}\special{pa 395 -18}%
\special{pa 394 0}%
\special{fp}%
\special{pa 2362 0}\special{pa 2361 -16}\special{pa 2359 -34}\special{pa 2355 -50}%
\special{pa 2350 -65}\special{pa 2342 -82}\special{pa 2334 -97}\special{pa 2323 -114}%
\special{pa 2311 -129}\special{pa 2299 -143}\special{pa 2284 -157}\special{pa 2269 -170}%
\special{pa 2253 -182}\special{pa 2234 -194}\special{pa 2215 -204}\special{pa 2194 -215}%
\special{pa 2173 -224}\special{pa 2152 -232}\special{pa 2128 -240}\special{pa 2105 -246}%
\special{pa 2080 -252}\special{pa 2055 -256}\special{pa 2031 -259}\special{pa 2005 -261}%
\special{pa 1981 -262}\special{pa 1957 -262}\special{pa 1931 -261}\special{pa 1907 -259}%
\special{pa 1881 -256}\special{pa 1856 -251}\special{pa 1833 -246}\special{pa 1808 -240}%
\special{pa 1786 -233}\special{pa 1765 -225}\special{pa 1744 -215}\special{pa 1724 -206}%
\special{pa 1704 -194}\special{pa 1686 -183}\special{pa 1670 -171}\special{pa 1654 -158}%
\special{pa 1640 -144}\special{pa 1626 -129}\special{pa 1614 -114}\special{pa 1604 -99}%
\special{pa 1595 -83}\special{pa 1588 -67}\special{pa 1582 -51}\special{pa 1578 -34}%
\special{pa 1576 -18}\special{pa 1575 0}%
\special{fp}%
\special{pa 1240 -59}\special{pa 1271 -59}\special{fp}\special{pa 1301 -59}\special{pa 1332 -59}\special{fp}%
\special{pa 1363 -59}\special{pa 1393 -59}\special{fp}\special{pa 1424 -59}\special{pa 1455 -59}\special{fp}%
\special{pa 1485 -59}\special{pa 1516 -59}\special{fp}%
\special{pa 1516 -59}\special{pa 1516 -64}\special{pa 1516 -70}\special{pa 1516 -75}\special{pa 1517 -80}\special{pa 1518 -86}\special{pa 1519 -91}\special{pa 1519 -92}\special{fp}%
\special{pa 1528 -125}\special{pa 1529 -128}\special{pa 1532 -133}\special{pa 1534 -138}\special{pa 1537 -143}\special{pa 1539 -148}\special{pa 1542 -153}\special{pa 1543 -155}\special{fp}%
\special{pa 1562 -182}\special{pa 1562 -183}\special{pa 1566 -187}\special{pa 1570 -192}\special{pa 1574 -196}\special{pa 1578 -201}\special{pa 1583 -206}\special{pa 1585 -207}\special{fp}%
\special{pa 1610 -229}\special{pa 1612 -231}\special{pa 1617 -235}\special{pa 1623 -239}\special{pa 1628 -243}\special{pa 1634 -246}\special{pa 1637 -249}\special{fp}%
\special{pa 1666 -266}\special{pa 1670 -268}\special{pa 1676 -271}\special{pa 1683 -274}\special{pa 1690 -277}\special{pa 1697 -280}\special{fp}%
\special{pa 1633 -181}\special{pa 1651 -209}\special{fp}\special{pa 1669 -237}\special{pa 1688 -266}\special{fp}%
\special{pa 1680 -279}\special{pa 1646 -277}\special{fp}\special{pa 1612 -275}\special{pa 1579 -273}\special{fp}%
\special{pn 24}%
\special{pa -197 0}\special{pa 2559 0}%
\special{fp}%
\special{pn 8}%
\special{pn 8}%
\special{pa 656 -175}\special{pa 787 -262}\special{pa 640 -317}%
\special{fp}%
\special{pn 8}%
\special{pn 8}%
\special{pa 1837 -175}\special{pa 1968 -262}\special{pa 1821 -317}%
\special{fp}%
\special{pn 8}%
\settowidth{\Width}{$c$}\setlength{\Width}{-0.5\Width}%
\settoheight{\Height}{$c$}\settodepth{\Depth}{$c$}\setlength{\Height}{\Depth}%
\put(1.0000,-0.4500){\hspace*{\Width}\raisebox{\Height}{$c$}}%
\settowidth{\Width}{$\overline{c}$}\setlength{\Width}{-0.5\Width}%
\settoheight{\Height}{$\overline{c}$}\settodepth{\Depth}{$\overline{c}$}\setlength{\Height}{\Depth}%
\put(3.0000,-0.4500){\hspace*{\Width}\raisebox{\Height}{$\overline{c}$}}%
\settowidth{\Width}{$d$}\setlength{\Width}{-0.5\Width}%
\settoheight{\Height}{$d$}\settodepth{\Depth}{$d$}\setlength{\Height}{\Depth}%
\put(4.0000,-0.4500){\hspace*{\Width}\raisebox{\Height}{$d$}}%
\settowidth{\Width}{$\overline{d}$}\setlength{\Width}{-0.5\Width}%
\settoheight{\Height}{$\overline{d}$}\settodepth{\Depth}{$\overline{d}$}\setlength{\Height}{\Depth}%
\put(6.0000,-0.4500){\hspace*{\Width}\raisebox{\Height}{$\overline{d}$}}%
\settowidth{\Width}{$c_i$}\setlength{\Width}{0.\Width}%
\settoheight{\Height}{$c_i$}\settodepth{\Depth}{$c_i$}\setlength{\Height}{\Depth}%
\put(2.0500,0.8167){\hspace*{\Width}\raisebox{\Height}{$c_i$}}%
\settowidth{\Width}{$c_j$}\setlength{\Width}{0.\Width}%
\settoheight{\Height}{$c_j$}\settodepth{\Depth}{$c_j$}\setlength{\Height}{\Depth}%
\put(5.0500,0.8167){\hspace*{\Width}\raisebox{\Height}{$c_j$}}%
\end{picture}}%
\caption{A chord slide for the opposite position
\label{opposite_direction}}
\end{center}
\end{figure}
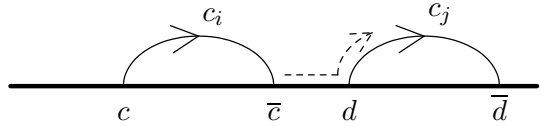

For $i\neq j$,
let $M_s(\alpha,\beta)$ and $M_o(\alpha,\beta)$ denote the 
$2g\times 2g$ matrices for which the $(i,j)$-th entries are
$-1$ and $1$ respectively for $(i,j)=(\alpha,\beta)$
and $\delta_{i,j}$ for $(i,j)\neq (\alpha,\beta)$.
Here $\delta_{i,j}$ is Kronecker's delta. 
After a chord slide of $c_i$ along $c_j$ for the same position,
we have 
\[(c^{\prime}_1,c^{\prime}_2,\ldots,c^{\prime}_{2g})
=(c_1,c_2,\ldots,c_{2g}){}^tM_s(i,j).
\]
For the opposite position, ${}^tM_s(i,j)$ may be replaced with ${}^tM_o(i,j)$.
Let $s_{i,j}(A)$ and $o_{i,j}(A)$ denote, respectively, matrices 
$M_s(i,j) A {}^tM_s(i,j) $ and $M_o(i,j) A {}^tM_o(i,j)$.
By noting the changes in the intersection numbers of the chords $(c_1,c_2,\ldots,c_{2g})$,
we have
\begin{prop}
Consider loops $\{c_k\}_{k=1,2,\ldots,2g}$ as chords in the linear chord diagram
and $A$ its intersection matrix.
If we slide $c_i$ along $c_j$ for the same position,
the intersection matrix of $\{c_k^{\prime}\}_{k=1,2,\ldots,2g}$ is $s_{i,j}(A)$.
For the opposite position, it is $o_{i,j}(A)$.
\end{prop}

Using this proposition, we have only to deform the intersection matrix 
into a $2g\times 2g$ symplectic matrix
$
\left(
\begin{array}{cc}
O & I_g \\
-I_g & O
\end{array}
\right)
$.
Here $I_g$ is the identity matrix of size $g$.

For the Klein quartic $C_{7,2}$,
the chord slide algorithm yields the matrices
\[s_{6,4}\circ o_{6,2}\circ o_{5,2}\circ o_{5,1}\circ s_{2,1}(A)\]
and
\[M=M_{s}(6,4)M_{o}(6,2)M_{o}(5,2)M_{s}(5,3)M_{o}(5,1)M_{s}(2,1),\]
or explicitly
\begin{center}
$\left(
 \begin{array}{cccccc}
 0 & 0 & 1 & 0 & 0 & 0 \\
 0 & 0 & 0 & 1 & 0 & 0 \\
-1 & 0 & 0 & 0 & 0 & 0 \\
 0 &-1 & 0 & 0 & 0 & 0 \\
 0 & 0 & 0 & 0 & 0 & 1 \\
 0 & 0 & 0 & 0 &-1 & 0 
 \end{array}
\right)$
and
$\left(
 \begin{array}{cccccc}
 1 & 0 & 0 & 0 & 0 & 0 \\
-1 & 1 & 0 & 0 & 0 & 0 \\
 0 & 0 & 1 & 0 & 0 & 0 \\
 0 & 0 & 0 & 1 & 0 & 0 \\
 0 & 1 &-1 & 0 & 1 & 0 \\
-1 & 1 & 0 &-1 & 0 & 1 
 \end{array}
\right),$
\end{center}
respectively.
We denote $(c_1,c_2,\ldots,c_6){}^{t}M$ by $(c^{\prime}_1,c^{\prime}_2,\ldots,c^{\prime}_6)$ such that
$c_3^{\prime}=c_5$, $c_4^{\prime}=c_3$, and $c_5^{\prime}=c_4$.
The resulting matrix is denoted by $T_{7,2}$.
We have the following matrices $T_{7,2}$ and $T_{7,2}A{}^{t}T_{7,2}$,
\[
T_{7,2}=
\left(
 \begin{array}{cccccc}
 1 & 0 & 0 & 0 & 0 & 0 \\
-1 & 1 & 0 & 0 & 0 & 0 \\
 0 & 1 &-1 & 0 & 1 & 0 \\
 0 & 0 & 1 & 0 & 0 & 0 \\
 0 & 0 & 0 & 1 & 0 & 0 \\
-1 & 1 & 0 &-1 & 0 & 1 
 \end{array}
\right),\ 
T_{7,2}A{}^{t}T_{7,2}=
\left(
 \begin{array}{cccccc}
 0 & 0 & 0 & 1 & 0 & 0 \\
 0 & 0 & 0 & 0 & 1 & 0 \\
 0 & 0 & 0 & 0 & 0 & 1 \\
-1 & 0 & 0 & 0 & 0 & 0 \\
 0 &-1 & 0 & 0 & 0 & 0 \\
 0 & 0 &-1 & 0 & 0 & 0 
 \end{array}
\right),
\]
from which we then obtain a symplectic basis
$(a_1,a_2,a_3,b_1,b_2,b_3)=(c_1,c_2,\ldots,c_6){}^{t}T_{7,2}$
of $H_1(C_{7,2};\Z)$.
The matrix $T_{7,2}$ is different from that given in \cite{0557.30036}.

\subsection{Period matrices}
\label{Period matrix}
We introduce a method for calculating the period matrix of
$X$ using holomorphic 1-forms of Bennama and Carbonne \cite{0842.14022}.
We compute those of $C_{7,1}$ and $C_{7,2}$.

Let $H^{1,0}(X)$ be the space of holomorphic 1-forms on $X$.
The floor function is denoted by $\lfloor \cdot \rfloor$.
For $n=1,2,\ldots, p-1(=2g)$, we define $\alpha_l$ and $\alpha_m$
by $\left\lfloor\dfrac{nl}{p}\right\rfloor$ and
$\left\lfloor\dfrac{nm}{p}\right\rfloor$, respectively.
Set $d_n=\left\lfloor\dfrac{n(l+m)}{p} \right\rfloor-\alpha_l -\alpha_m -1$.
Bennama and Carbonne \cite{0842.14022} derived a basis for $H^{1,0}(X)$ 
\begin{center}
$\displaystyle
\omega_{n,d}=\dfrac{x^{\alpha_l}(1-x)^{\alpha_m}x^d}{{y^n}}$
with $0\leq d\leq d_n$ and $1\leq n\leq p-1$.
\end{center}
Take a symplectic basis $\{a_i,b_i\}_{i=1,2,\ldots,g}$ of $H_{1}(X;\Z)$,
{\it i.e.}, their intersection numbers are $a_i\cdot b_j=\delta_{i,j}$
and $a_i\cdot a_j=b_i\cdot b_j=0$.
We define two $g\times g$ matrices $\Omega_A$ and $\Omega_B$ by $\left(\int_{a_j}\omega_i\right)$
and $\left(\int_{a_j}\omega_i\right)$.
It is known that the period matrix
with respect to $\{a_i,b_i\}_{i=1,2,\ldots,g}$
is obtained by $\Omega_A^{-1}\Omega_B$; 
see \cite{0758.30002}, for example.

We illustrate with curves $C_{7,1}$ and $C_{7,2}$.
For $k=1,2$, we define a basis 
$\eta_{1}^{k}, \eta_{2}^{k}, \eta_{3}^{k}$
of $H^{1,0}(C_{7,k})$ as follows:
\[
 \begin{array}{c|ccc}
  & \eta_{1}^{k} & \eta_{2}^{k} & \eta_{3}^{k}\\ \hline
  C_{7,1} & \dfrac{dx}{y^6} & \dfrac{dx}{y^5}& \dfrac{dx}{y^4} \\ [8pt] \hline
  C_{7,2} & \dfrac{(1-x)dx}{y^6} & \dfrac{(1-x)dx}{y^5}& \dfrac{dx}{y^3} 
 \end{array}\ .
\]
Let $B(u,v)$ denote the beta function $\displaystyle\int_0^1t^{u-1}
(1-t)^{v-1}dt$ for $u,v>0$. 
Put $(h_1,h_2,h_3,h_4)=(1/{7},2/{7},4/{7},1/{7})$.
From
\[\int_{I_0}\eta_{i}^{k}=
\left\{
 \begin{array}{ll}
  B(i/{7}, i/{7}) & (k=1)\\
  B(h_i, h_{i+1}) & (k=2)
 \end{array}
\right.,\]
the holomorphic 1-forms $\omega_{i}^{1}$ and $\omega_{i}^{2}$ are denoted by 
$\eta_{i}^{1}/{B(i/{7}, i/{7})}$ and $\eta_{i}^{2}/{B(h_i, h_{i+1})}$, respectively.
The equation
\[\int_{c_j}\omega_{i}^{k}=\int_{I_0}\omega_{i}^{k}-\int_{\sigma_{\ast}^{j}I_0}\omega_{i}^{k}=\int_{I_0}\omega_{i}^{k}-\int_{I_0}(\sigma^{\ast})^{j}\omega_{i}^{k}\]
yields
\begin{lem}
\label{period}
\[
\int_{c_j}\omega_{i}^{k}=
\left\{
\begin{array}{cl}
1-\zeta_7^{ij} & (k=1), \\ [6pt]
1-\zeta_7^{7h_i j} & (k=2).
\end{array}
\right.
\]
\end{lem}
\begin{rem}
These integrals depend only on the cohomology class of $\omega_{i}^{k}$
and the homology class of $c_j$.
\end{rem}

For $k=1,2$, let $A_k$ and $B_k$ be the $3\times 3$ matrices
$\left(\int_{a_j}\omega_i^k\right)$ and
$\left(\int_{b_j}\omega_i^k\right)$, respectively.
We construct the $3\times 6$ matrices
\[
(A_k, B_k)=\left(\int_{c_j}\omega_{i}^{k}\right){}^{t}T_{7,k}.
\]
Here, for the case $p=7$, we denote $T_{7,1}=T$, the matrix (\ref{T_{7,1}})
in Remark \ref{row vector representation}.
We obtain the period matrix for $C_{7,k}$.
\begin{prop}
Let $\tau_{7,k}$ be the period matrix $A_k^{-1}B_k$ of $C_{7,k}$.
Then we have
\[
\tau_{7,1}=
\left(
\begin{array}{ccc}
 4+\zeta+2 \zeta^2+2 \zeta^3+\zeta^4+2 \zeta^5 & -1-2 \zeta-2 \zeta^3-2 \zeta^5 & -1+\zeta-\zeta^2 \\
 -1-2 \zeta-2 \zeta^3-2 \zeta^5 & -1+\zeta-\zeta^2-\zeta^4-\zeta^5 & 1+\zeta^3+\zeta^5 \\
 -1+\zeta-\zeta^2 & 1+\zeta^3+\zeta^5 & 1+\zeta^2
\end{array}
\right)\]
and
\[
\tau_{7,2}=
\dfrac{1}{4}
\left(
\begin{array}{ccc}
 6+3\xi & 4+2\xi & -2-\xi \\
 4+2\xi & 4+4\xi & -2\xi \\
 -2-\xi & -2\xi & 2+3\xi
\end{array}
\right),
\]
where $\zeta=\zeta_7$ and $\xi=\zeta+\zeta^2+\zeta^4=(-1+\sqrt{-7})/2$.
\end{prop}
\begin{proof}
By definition of $A_k$ and $B_k$, we have matrices
\begin{align*}
  A_1&=\left(
\begin{array}{ccc}
 1-\zeta & 1-\zeta+\zeta^2-\zeta^3 & 1-\zeta+\zeta^2-\zeta^3+\zeta^4-\zeta^5 \\
 1-\zeta^2 & 1-\zeta^2+\zeta^4-\zeta^6 & 1-\zeta^2+\zeta^4-\zeta^6+\zeta^8-\zeta^{10} \\
 1-\zeta^3 & 1-\zeta^3+\zeta^6-\zeta^9 & 1-\zeta^3+\zeta^6-\zeta^9+\zeta^{12}-\zeta^{15}
\end{array}
\right),\\
B_1&=\left(
\begin{array}{ccc}
 1-\zeta^2 & 1-\zeta+\zeta^2-\zeta^4 & 1-\zeta+\zeta^2-\zeta^3+\zeta^4-\zeta^6 \\
 1-\zeta^4 & 1-\zeta^2+\zeta^4-\zeta^8 & 1-\zeta^2+\zeta^4-\zeta^6+\zeta^8-\zeta^{12} \\
 1-\zeta^6 & 1-\zeta^3+\zeta^6-\zeta^{12} & 1-\zeta^3+\zeta^6-\zeta^9+\zeta^{12}-\zeta^{18}
\end{array}
\right),\\
 A_2&=
  \left(
\begin{array}{ccc}
 1-\zeta & \zeta-\zeta^2 & 1-\zeta^2+\zeta^3-\zeta^5 \\
 1-\zeta^2 & \zeta^2-\zeta^4 & 1-\zeta^4+\zeta^6-\zeta^{10} \\
 1-\zeta^4 & \zeta^4-\zeta^8 & 1-\zeta^8+\zeta^{12}-\zeta^{20}
\end{array}
\right)
,\\
 B_2&=
  \left(
\begin{array}{ccc}
 1-\zeta^3 & 1-\zeta^4 & \zeta-\zeta^2+\zeta^4-\zeta^6 \\
 1-\zeta^6 & 1-\zeta^8 & \zeta^2-\zeta^4+\zeta^8-\zeta^{12} \\
 1-\zeta^{12} & 1-\zeta^{16} & \zeta^4-\zeta^8+\zeta^{16}-\zeta^{24}
\end{array}
\right).
\end{align*}
The determinants of $A_1$ and $A_2$ are $-7(\zeta^4+\zeta^5)$
and $7(1-\xi)$, respectively.
Using the adjoint matrices of $A_1$ and $A_2$, we obtain their inverse matrices
{\scriptsize
\begin{align*}
 A_1^{-1}&=-\dfrac{1+\zeta^2+\zeta^3+\zeta^5}{7}\\
&\left(
\begin{array}{ccc}
 -\zeta+3 \zeta^2+4 \zeta^3+\zeta^4 & -1+3 \zeta-\zeta^2+2 \zeta^3-\zeta^4-2 \zeta^5 & -3-2 \zeta-2 \zeta^2-\zeta^3-4 \zeta^4-2 \zeta^5 \\
 1+4 \zeta+3 \zeta^2-\zeta^3 & 1-\zeta+3 \zeta^3+4 \zeta^5 & -\zeta+2 \zeta^2-2 \zeta^4+\zeta^5 \\
 -1-3 \zeta-4 \zeta^2-2 \zeta^3-2 \zeta^4-2 \zeta^5 & 1+2 \zeta^2-2 \zeta^3-\zeta^5 & 2+\zeta+2 \zeta^3+3 \zeta^4-\zeta^5
\end{array}
\right)
,\\
   A_2^{-1}&=\dfrac{2+\xi}{28}
\left(
\begin{array}{ccc}
 3+\zeta-\zeta^3-3 \zeta^4 & 4-2 \zeta+2 \zeta^2+\zeta^3+\zeta^4+\zeta^5 & 3-3 \zeta^2+\zeta^4-\zeta^5 \\
 -\zeta-\zeta^2-5 \zeta^3-4 \zeta^4-3 \zeta^5 & 5+\zeta+4 \zeta^2+2 \zeta^3+4 \zeta^4+5 \zeta^5 & 3+2 \zeta-\zeta^2+3 \zeta^3+2 \zeta^4-2 \zeta^5 \\
 -2 \zeta-\zeta^2+\zeta^3+2 \zeta^4 & -1+\zeta-3 \zeta^2-\zeta^3-2 \zeta^4-\zeta^5 & -\zeta+2 \zeta^2-2 \zeta^4+\zeta^5
\end{array}
\right)
.
\end{align*}
}
It suffices to calculate $A_1^{-1}B_1$ and $A_2^{-1}B_2$.
\end{proof}

\section{Two symplectic bases of one family of hyperelliptic curves}
\label{Two symplectic basis of a plain curve}
For {\it odd} integer $q\geq 5$,
let $C_{q,1}$ be a plane algebraic curve defined by the affine equation
$y^{q}=x(1-x)$.
We obtain two symplectic bases for the first integral homology group
$H_{1}(C_{q,1};\Z)$.
By substituting
$y=\sqrt[q]{\dfrac{1}{4}}\, z$ and $x=\dfrac{\sqrt{-1}\,w-1}{2}$
into the above equation, we have $w^2=z^{q}-1$,\
corresponding to a hyperelliptic curve of genus $g=(q-1)/2$.
Set the order-$q$ holomorphic automorphism
$\sigma(x,y)=(x,\zeta y)$ with $\zeta=\zeta_q=\exp(2\pi\sqrt{-1}/{q})$.
In terms of parameters $z$ and $w$, we define a loop
$\gamma_{k}\colon [0,1]\to C_{q,1}, k=0,1,\ldots, 2g,$ by
\[
\gamma_{k}(t)
=
\left\{
\begin{array}{ll}
 (\zeta^k\cdot 2t, \sqrt{-1}\sqrt{1-(2t)^p})& (0\leq t\leq 1/2),\\
 (\zeta^k(2-2t), -\sqrt{-1}\sqrt{1-(2-2t)^p})& (1/2\leq t\leq 1).
\end{array}
\right.
\]
We define the path $I_0\colon [0,1]\to C_{q,1}$
in a similar manner as in Section \ref{Dessins d'enfants}.
It is easy to prove
\begin{lem}
\label{homotopy with relative endpoints}
For $k=0,1,\ldots, 2g$, the two paths $(\sigma_{\ast})^{k}I_0$
and $\gamma_{k}$ are homotopic with relative endpoints.
\end{lem}

We recall the following well-known fact (see \cite{0192.58201}, for example).
\begin{prop}
For $i=1,2,\ldots, g$, we denote $A_i=\gamma_{2i-1}\cdot \gamma_{2i}^{-1}$ and
$B_i=\gamma_{2i-1}\cdot \gamma_{2i-2}^{-1}\cdot \cdots \cdot \gamma_{1}\cdot \gamma_{0}^{-1}$.
We then have 
$\{A_i,B_i\}_{i=1,2,\ldots,g}$ is a symplectic basis of $H_{1}(C_{q,1}; \Z)$.
\end{prop}
We call this basis a natural type.
Indeed, this proposition immediately follows from a two-sheeted covering
$C_{q,1}\ni (z,w)\to z\in \C P^1$ branched over the $2g+2$ points
$\{1,\zeta,\zeta^2,\ldots,\zeta^{2g},\infty\}\subset \C P^1$.
We find another symplectic basis of $H_{1}(C_{q,1}; \Z)$.
Although in general $q$ is not prime,
the chord slide algorithm can be similarly applied to $C_{q,1}$.
We recall $c_i=I_0\cdot (\sigma_{\ast})^{i}I_{0}^{-1}$.
Remark \ref{intersection numbers} gives us the intersection numbers of $c_i$'s
\[
c_i\cdot c_j=
\left\{
 \begin{array}{cl}
 1 & (i<j),\\
 0 & (i=j),\\
-1 & (i>j).
 \end{array}
\right.
\]
For $k=1,2,\ldots,g-1$ and $i=1,2,\ldots,2g$,
let $h_{i,k}$ be the composition $o_{i,2k-1}\circ s_{i,2k}$.
Moreover, we define the move $f_k$ by
\[
h_{2g,k}\circ h_{2g-1,k}\circ \cdots \circ h_{2k+2,k}\circ h_{2k+1,k}.
\]

\def\hugesymbol#1{\mbox{\strut\rlap{\smash{\LARGE$#1$}}\quad}}
\begin{lem}
\label{symplectic matrix}
Let $A=(c_i\cdot c_j)$ be the intersection matrix of $c_i$'s.
Then, the $2g\times 2g$ matrix
\[f_{g-1}\circ f_{g-2}\circ \cdots \circ f_2 \circ f_1(A)\]
is equal to the $2g\times 2g$ matrix
\[J_g=\left( \begin{array}{cccc}
    J              &     &        & \hugesymbol{O} \\
                   & J   &        &                \\
                   &     & \ddots &                \\
    \hugesymbol{O} &     &        & J
  \end{array} \right),\]
where $J=\left( \begin{array}{cc}
    0 & 1 \\  -1 & 0
   \end{array} \right)$.
\end{lem}
\begin{proof}
The intersection matrices of $c_i$'s in Figures \ref{Start},
\ref{second}, \ref{middle}, and \ref{Goal}
represent $A$, $h_{3,1}(A)$, $f_1(A)$, and $f_{g-1}\circ f_{g-2}\circ \cdots \circ f_2 \circ f_1(A)$ respectively.
We illustrate each move $f_i$, in particular $f_1$.
We place emphasis on the endpoint series in the linear chord diagrams.
The endpoint series in Figure \ref{Start} is 
\[1,2,3,\ldots,2g-1,2g,\overline{1},\overline{2},\overline{3},\ldots,\overline{2g-1},\overline{2g}.\]
In this figure, we first take chord slides $\overline{3}\to \overline{2}$ and $\overline{3}\to 1$.
Figure \ref{second} is obtained and
the intersection matrix becomes $h_{3,1}(A)$.
Similarly, we take chord slides
\[\overline{4}\to \overline{2}, \overline{4}\to 1,
\overline{5}\to \overline{2}, \overline{5}\to 1,\ldots,
\overline{2g}\to \overline{2}, \overline{2g}\to 1.\]
We obtain Figure \ref{middle} and the intersection matrix becomes $f_{1}(A)$.
The endpoints series in this figure is
\[1,2,3,\ldots,2g-1,2g,\overline{3},\overline{4},\overline{5},\ldots,\overline{2g-1},\overline{2g},\overline{1},\overline{2}.\]
For each $k$-th step, $k=2,3,\ldots,g-1$,
we take chord slides
\[\overline{2k+1}\to \overline{2k}, \overline{2k+1}\to 2k-1,
\overline{2k+2}\to \overline{2k}, \overline{2k+2}\to 2k-1,\ldots,
\overline{2g}\to \overline{2k}, \overline{2g}\to 2k-1.\]
This step corresponds to $f_k$.
Finally, we have Figure \ref{Goal} and
the intersection matrix becomes
$f_{g-1}\circ f_{g-2}\circ \cdots \circ f_2 \circ f_1(A)$.
The endpoint series in this figure is
\[1,2,3,\ldots,2g-1,2g,\overline{2g-1},\overline{2g},\overline{2g-3},\overline{2g-2},\ldots,\overline{3},\overline{4},\overline{1},\overline{2}.\]
Clearly, the intersection matrix $f_{g-1}\circ f_{g-2}\circ \cdots \circ f_2 \circ f_1(A)$ is equal to $J_{g}$ from Figure \ref{Goal}.
\end{proof}

\begin{figure}[htbp]
\begin{center}
\scalebox{0.7}{\input{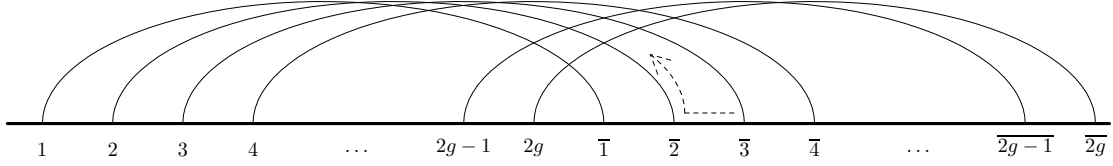}}
\caption{The initial linear chord diagram}\label{Start} 
\end{center}
\end{figure}

\begin{figure}[htbp]
\begin{center}
\scalebox{0.7}{\input{linear-chord-diagram_proof_12}}
\caption{After chord slides $\overline{3}\to \overline{2}$ and $\overline{3}\to 1$ in Figure \ref{Start}}
\label{second}
\end{center}
\end{figure}

\begin{figure}
\begin{center}
\scalebox{0.7}{\input{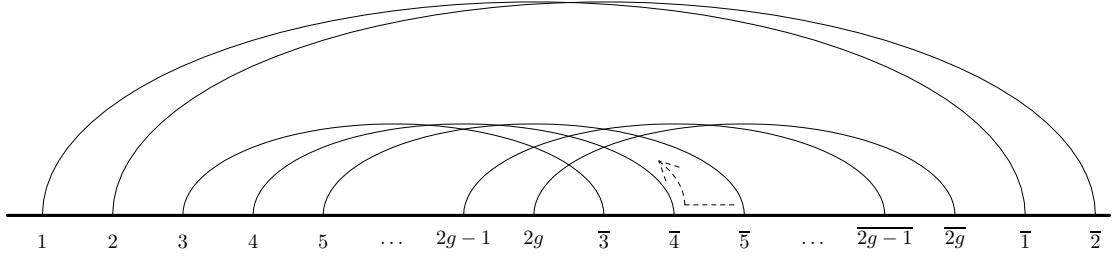}}
\caption{After the first step in Figure \ref{Start}}
\label{middle}
\end{center}
\end{figure}

\begin{figure}
\begin{center}
\scalebox{0.65}{\input{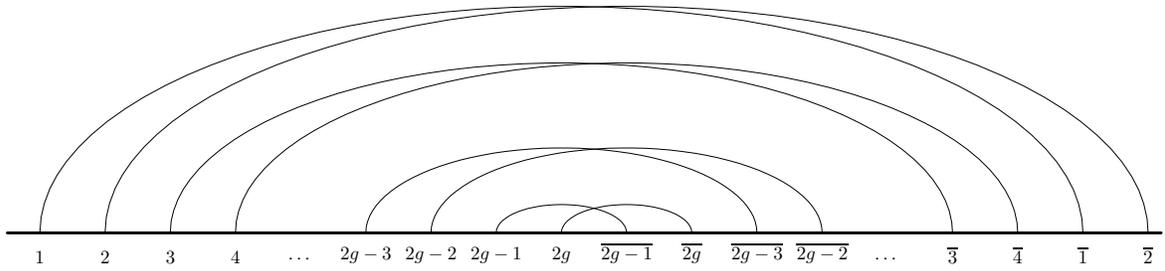}}
\caption{The final linear chord diagram}\label{Goal}
\end{center}
\end{figure}

For $g=3$, we give explicit forms for the intersection matrices
$A$, $f_1(A)$, and $f_2\circ f_1(A)$
corresponding to Figure \ref{Start}, \ref{middle}, and
\ref{Goal}, respectively.

\begin{center}
\begin{minipage}[cbt]{.4\textwidth}
\begin{center}
$\left(
\begin{array}{cccccc}
 0 & 1 & 1 & 1 & 1 & 1 \\
 -1 & 0 & 1 & 1 & 1 & 1 \\
 -1 & -1 & 0 & 1 & 1 & 1 \\
 -1 & -1 & -1 & 0 & 1 & 1 \\
 -1 & -1 & -1 & -1 & 0 & 1 \\
 -1 & -1 & -1 & -1 & -1 & 0 \\
\end{array}
\right),
$
\end{center}
\end{minipage}
\ 
\begin{minipage}[cbt]{.5\textwidth}
\begin{center}
$
\left(
\begin{array}{cccccc}
 0 & 1 & 0 & 0 & 0 & 0 \\
-1 & 0 & 0 & 0 & 0 & 0 \\
 0 & 0 & 0 & 1 & 1 & 1 \\
 0 & 0 &-1 & 0 & 1 & 1 \\
 0 & 0 &-1 &-1 & 0 & 1 \\
 0 & 0 &-1 &-1 &-1 & 0 \\
\end{array}
\right),
$
\end{center}
\end{minipage}
\end{center}
\begin{center}
and
$
\left(
\begin{array}{cccccc}
 0 & 1 & 0 & 0 & 0 & 0 \\
 -1 & 0 & 0 & 0 & 0 & 0 \\
 0 & 0 & 0 & 1 & 0 & 0 \\
 0 & 0 & -1 & 0 & 0 & 0 \\
 0 & 0 & 0 & 0 & 0 & 1 \\
 0 & 0 & 0 & 0 & -1 & 0 \\
\end{array}
\right).
$
\end{center}

\begin{rem}
We obtain the matrix $T^{\prime}$ such that
$f_{g-1}\circ f_{g-2}\circ \cdots \circ f_2 \circ f_1(A)
=T^{\prime}A{}^{t}T^{\prime}$
\[
T^{\prime}=
\left(
\begin{array}{ccccccccc}
 1 & 0 & 0 & 0 & \cdots& & 0 & 0 \\
 0 & 1 & 0 & 0 & \cdots& & 0 & 0 \\
 1 &-1 & 1 & 0 & \cdots& & 0 & 0 \\
 1 &-1 & 0 & 1 &       & & 0 & 0 \\
 1 &-1 & 1 &-1 & \ddots& & \vdots & \vdots  \\
\vdots & \vdots& \vdots& \vdots& & & &  \\
   &   &   &   &       & & 1 & 0 \\
 1 &-1 & 1 &-1 &       & & 0 & 1 
\end{array}
\right).
\]
\end{rem}

Let $c_1^{\prime},c_2^{\prime},\ldots, c_{2g}^{\prime}$ be a basis of $H_1(C_{q,1}; \Z)$ obtained by
Lemma \ref{symplectic matrix}.
Interchanging this basis gives us a symplectic basis 
\[(a_1,\ldots,a_g,b_1,\ldots, b_g)=
(c_1^{\prime},c_3^{\prime},\ldots, c_{2g-1}^{\prime},
c_2^{\prime},c_4^{\prime},\ldots, c_{2g}^{\prime}).\]

We interchange the rows of the matrix $f_{g-1}\circ f_{g-2}\circ \cdots \circ f_2 \circ f_1(A)$
in the following way.
All odd rows move to rows $1,2,\ldots, g$-th and all even rows to rows $g+1,g+2,\ldots,2g$.
The resulting matrix is denoted by $T$.
\begin{thm}
Set $(a_1,\ldots,a_g,b_1,\ldots, b_g)=(c_1,c_2,\ldots,c_{2g}){}^tT$.
Then, $\{a_i,b_i\}_{i=1,2,\ldots,g}$ is a symplectic basis of $H_1(C_{q,1}; \Z)$.
\end{thm}
\begin{rem}
\label{row vector representation}
The matrix $T$ takes the explicit form
\[
T=
\left(
\begin{array}{cccccccccccc}
 1 & 0 &   &   &  & \cdots  &   &  &&& 0 \\
 1 &-1 & 1 & 0 &   &  & \cdots  &  &&& 0 \\
 1 &-1 & 1 &-1 & 1 & 0 &  & \cdots &&& 0 \\
   &   &   &   &   &\vdots&   &   &   &   &   \\
 1 &-1 & 1 &-1 & 1 & -1& 1 & \cdots &-1 & 1 & 0 \\
 0 & 1 & 0 &  & & \cdots & & & & & 0 \\
 1 &-1 & 0 & 1 & 0 & & & \cdots& && 0 \\
 1 &-1 & 1 & -1& 0 & 1& 0 & & \cdots & & 0 \\
   &   &   &   &   &\vdots&   &   &   &   &   \\
 1 &-1 & 1 & -1& 1 &-1& & \cdots &  -1& 0& 1 \\
\end{array}
\right).
\]

For $g=3$, the matrix $T=T_{7,1}$ is given by
\begin{equation}
\label{T_{7,1}}
\left(
\begin{array}{cccccc}
 1 & 0 & 0 & 0 & 0 & 0 \\
 1 &-1 & 1 & 0 & 0 & 0 \\
 1 &-1 & 1 &-1 & 1 & 0 \\
 0 & 1 & 0 & 0 & 0 & 0 \\
 1 &-1 & 0 & 1 & 0 & 0 \\
 1 &-1 & 1 &-1 & 0 & 1 \\
\end{array}
\right).
\end{equation}
\end{rem}
We prove that the two symplectic bases $\{A_i,B_i\}_{i=1,2,\ldots,g}$
and $\{a_i,b_i\}_{i=1,2,\ldots,g}$ are different.
From Lemma \ref{homotopy with relative endpoints},
we have matrix $K$
\[
K=
\left(
\begin{array}{cccccccc}
-1 & 1 & 0 &   &\cdots&   & 0 \\
 0 & 0 &-1 & 1 & 0 &\cdots& 0 \\
   &   &   &\vdots&   &   &   \\
 0 &  &\cdots& & 0 &-1 & 1 \\
-1 & 0 &     &\cdots& & & 0 \\
-1 & 1 &-1 & 0 &\cdots&& 0 \\
   &   &  &\vdots  &   &   &   \\
-1 & 1 &-1 & 1 &\cdots&-1 & 0 
\end{array}
\right)
\]
such that $(A_1,\ldots,A_g,B_1,\ldots, B_g)=(c_1,c_2,\ldots,c_{2g}){}^tK$.
By comparing $T$ and $K$, we have
\begin{prop}
The two symplectic bases $\{A_i,B_i\}_{i=1,2,\ldots,g}$
and $\{a_i,b_i\}_{i=1,2,\ldots,g}$ are not equal up to ordering.
\end{prop}

\section{Period matrices for hyperelliptic curves}
\label{Period matrices for hyperelliptic curves}
Let $\tau_g$ and $\tau_g^{\mathrm{CS}}$ denote the period matrices
of $C_{q,1}$ with respect to symplectic bases $\{A_i,B_i\}_{i=1,2,\ldots,g}$ and $\{a_i,b_i\}_{i=1,2,\ldots,g}$ of $H_1(C_{q,1};\Z)$. 
We compute $\tau_g$ of $C_{q,1}$.
Moreover, we obtain the relations
among $\tau_g$, $\tau_g^{\mathrm{CS}}$, and Schindler's period matrices \cite{0801.14008}.
Set $\eta_i=\dfrac{dx}{y^{q-i}}$ for $i=1,2,\ldots,g$.
Bennama \cite{0931.14018} proved that $\{\eta_i\}_{i=1,2,\ldots,g}$
is a basis of $H^{1,0}(C_{q,1})$.
From Section \ref{Period matrix}, we have for the period of $\eta_i$ along $c_j$
\[
\int_{c_j}\eta_i=(1-\zeta^{ij})B(i/q,i/q).
\]
For simplicity, we denote $\omega_i=\eta_i/{B(i/q,i/q)}$.
We recall the two $g\times g$ matrices $\Omega_A$ and $\Omega_B$ which
are $\left(\int_{A_j}\omega_i\right)_{i,j}$
and $\left(\int_{B_j}\omega_i\right)$, respectively.
We form the $g\times 2g$ matrix,
\[(\Omega_A, \Omega_B)=\left(\int_{c_j}\omega_i\right){}^{t}K.\]
We have the periods of the matrices $\Omega_A$ and $\Omega_B$, for which
Tashiro, Yamazaki, Ito, and Higuchi \cite{0948.14501} obtained the same result.
\begin{prop}
We have
\begin{align*}
\Omega_A&=
-\mathop{\mathrm{diag}}\left(-1+\zeta^{i}\right)
\mathop{\mathrm{diag}}\left(\zeta^{i}\right)
\left(\zeta^{2i(j-1)}\right) {\text and}\\
\Omega_B&=\mathop{\mathrm{diag}}\left(-1+\zeta^{i}\right)
\left(\sum_{k=0}^{j-1}\zeta^{2ik}\right),
\end{align*}
where $\mathrm{diag}(a_{i})$ is the diagonal matrix with $(i,i)$-th entry $a_i$.
\end{prop}
\begin{proof}
From the definition of $\Omega_A$ and $\Omega_B$, we obtain
$\Omega_A=\left(\zeta^{i(2j-1)}-\zeta^{2ij}\right)$ and 
$\displaystyle \Omega_B=\left(\sum_{k=0}^{2j-1}(-1)^{k+1}\zeta^{ik}\right)$.
We have only to compute the $(i,j)$-th entries of $\Omega_A$ and $\Omega_B$
\begin{align*}
\zeta^{i(2j-1)}-\zeta^{2ij}
&=-\left(-1+\zeta^{i}\right)\zeta^{i}\zeta^{2i(j-1)},\\
\sum_{k=0}^{2j-1}(-1)^{k+1}\zeta^{ik}
&=-\dfrac{1-\zeta^{2ij}}{1+\zeta^{i}}
 =\left(-1+\zeta^{i}\right)\dfrac{1-\zeta^{2ij}}{1-\zeta^{2i}}
 =\left(-1+\zeta^{i}\right)\sum_{k=0}^{j-1}\zeta^{2ik}.
\end{align*}
\end{proof}
\begin{rem}
The matrix $\left(\zeta^{2i(j-1)}\right)$ is 
a Vandermonde matrix.
\end{rem}

To compute the matrix $\tau_g$ from $\Omega_A^{-1}\Omega_B$, we introduce a lemma.
For variables $x_1,x_2,\ldots, x_n$,
we denote by $\sigma_{i}(x_1,x_2,\ldots, x_n)$ the symmetric polynomial
\[
\sum_{1\leq j_1<\cdots < j_{i}\leq n} 
x_{j_1}\cdots x_{j_i}\]
for $1\leq i\leq n$ and $\sigma_{0}(x_1,x_2,\ldots, x_n)=1$.
Knuth \cite[Excercise 40 in \S 1.2.3]{0191.17903}
derived the inverse matrix of a Vandermonde matrix.
\begin{lem}
Let $a_1,a_2,\ldots,a_n$ be distinct complex constants.
We denote the Vandermonde matrix of size $n$ by
$V_{n}=\left(a_{i}^{j-1}\right)$.
Its inverse matrix is then
\[
V_{n}^{-1}=
\left((-1)^{i-1}\dfrac{\sigma_{n-i}(a_1,\ldots,\widehat{a_{j}},\ldots,a_{n})}
{\prod_{m=1, m\neq j}^{n}(a_m-a_j)}\right).
\]
The `hat' symbol is as defined earlier.
\end{lem}

The above proposition and lemma give 
\begin{thm}
The period matrix $\tau_g$ of $C_{q,1}$ with respect to the symplectic basis $\{A_i, B_i\}_{i=1,2,\ldots,g}$ is expressible as
\[
\tau_g=
\left(
\sum_{k=1}^{g}
\dfrac{(-1)^{i+g}}{2g+1}(1-\zeta^{2kj})
\sigma_{g-i}(\zeta^2,\zeta^4,
\ldots,\widehat{\zeta^{2j}},\ldots,\zeta^{2g})
\prod_{m=g-k+1}^{2g-k}(1-\zeta^{2m})
\right).
\]
\end{thm}
\begin{proof}
We compute $\Omega_A^{-1}\Omega_B$ as follows:
\begin{align*}
\Omega_A^{-1}\Omega_B&=-
\left(\zeta^{2i(j-1)}\right)^{-1}
\mathop{\mathrm{diag}}\left(\zeta^{-i}\right)
\left(\sum_{k=0}^{j-1}\zeta^{2ik}\right)\\
&=-
\left(\zeta^{2i(j-1)}\right)^{-1}
\left(\zeta^{-i}\dfrac{1-\zeta^{2ij}}{1-\zeta^{2i}}\right).
\end{align*}
From the equation
\[
\left(\zeta^{2i(j-1)}\right)^{-1}=
\left((-1)^{i-1}\dfrac{\sigma_{g-i}(\zeta^2,\zeta^4,
\ldots,\widehat{\zeta^{2j}},\ldots,\zeta^{2g})}
{\prod_{m=1, m\neq j}^{n}(\zeta^{2m}-\zeta^{2k})}\right),
\]
we have for the $(i,j)$-th entry of $\tau_g$
\[
\sum_{k=1}^{g}
(-1)^{i}\dfrac{\sigma_{g-i}(\zeta^2,\zeta^4,
\ldots,\widehat{\zeta^{2j}},\ldots,\zeta^{2g})}
{\prod_{m=1, m\neq k}^{g}(\zeta^{2m}-\zeta^{2k})}
\dfrac{1-\zeta^{2kj}}{\zeta^{k}(1-\zeta^{2k})}.
\]
Moreover, we obtain
\begin{align*}
\zeta^k(1-\zeta^{2k})\prod_{m=1, m\neq k}^{g}(\zeta^{2m}-\zeta^{2k})
&=\zeta^{k+2k(g-1)+2k}(\zeta^{-2k}-1)\prod_{m=1, m\neq k}^{g}(\zeta^{2(m-k)}-1)\\
&=(-1)^g\prod_{m=-k, m\neq 0}^{g-k}(1-\zeta^{2m})\\
&=(-1)^g\dfrac{2g+1}{\prod_{m=g-k+1}^{2g-k}(1-\zeta^{2m})},
\end{align*}
for each $k=1,2,\ldots,g$.
The last equality follows from
\[
\prod_{m=-k, m\neq 0}^{g-k}(1-\zeta^{2m})\prod_{m=g-k+1}^{2g-k}(1-\zeta^{2m})
=\prod_{l=1}^{2g}(1-\zeta^{l})=2g+1.
\]
This establishes the result.
\end{proof}

Setting $\zeta=\zeta_{7}$, we calculate $\tau_3$ of $C_{7,1}$ to be
\[
\left(
\begin{array}{ccc}
 -\zeta^5 & -1-\zeta^2-\zeta^4-\zeta^5 & 1+\zeta+\zeta^3+\zeta^5 \\
 -1-\zeta^2-\zeta^4-\zeta^5 & 1+\zeta+2 \zeta^3-\zeta^4+\zeta^5 & 2+\zeta^2+\zeta^3+\zeta^5 \\
 1+\zeta+\zeta^3+\zeta^5 & 2+\zeta^2+\zeta^3+\zeta^5 & 1+\zeta^2 \\
\end{array}
\right).\]

In general, the period matrix depends only on the choice of the symplectic basis
and the complex structure of the compact Riemann surface.
Two period matrices $\tau_g$ and $\tau_g^{\prime}$ are
obtained from the same compact Riemann surface
if and only if there exists a symplectic matrix
$
\left(
 \begin{array}{cc}
  P & Q \\
  R & S
 \end{array}
\right)\in \mathrm{Sp}(2g,\Z)
$
such that $\tau_g^{\prime}=(P+\tau_g R)^{-1}(Q+\tau_g S)$.
Here $P$, $Q$, $R$, and $S$ are $g\times g$ $\Z$-coefficient matrices.

The symplectic matrix
$
\left(
 \begin{array}{cc}
  P & Q \\
  R & S
 \end{array}
\right)
$
for the two period matrices of $C_{q,1}$ with respect to
$\{a_i,b_i\}_{i=1,2,\ldots,g}$ and $\{A_i,B_i\}_{i=1,2,\ldots,g}$
is given by $({}^{t}K)^{-1}{}^{t}T$, which we denote by $H$.
This matrix can be computed as
\[
H=\left(
 \begin{array}{cc}
  P & Q \\
  R & S
 \end{array}
\right)
=
\left(
\begin{array}{cc}
O & I_g \\
-I_g & -I_g \\
\end{array}
\right)\in \mathrm{Sp}(2g,\Z).
\]
Here $I_g$ is the identity matrix of size $g$.
The equation
\[\tau_g^{\mathrm{CS}}=(O+\tau_g\cdot(-I_g))^{-1}(I_g-\tau_g)\]
gives us
\begin{prop}
The relation between the two period matrices $\tau_g$ and $\tau_g^{\mathrm{CS}}$ is 
\[
\tau_g^{\mathrm{CS}}=-\tau_g^{-1}+I_g.
\]
\end{prop}
In particular, for $g=3$, we have $\tau_{7,1}=-\tau_3^{-1}+I_3$.
Nevertheless, the period matrix $\tau_g^{\mathrm{CS}}$ is complicated.

We introduce Schindler's period matrix, denoted by $\tau^{S}_{g}$,
for the hyperelliptic curve defined by the affine equation
$w_1^2=z_1(z_1^{2g+1}-1)$, and here denoted $C^{\prime}_{q,1}$.
This curve is biholomorphic to $C_{q,1}$.
For $i=1,2,\ldots,g$, elements $t_{i}$ of the $q$-th cyclotomic field $\Q(\zeta)$
are defined as follows:
\[
\left.
 \begin{array}{rcll}
  t_1&=&(-1)^g\zeta^{g^2},\\
  t_2&=&t_{1}\left(1-\dfrac{1}{1+\zeta}\right),\\
  t_{i+1}&=&\dfrac{t_{1}(1-\sum_{k=2}^{i}\zeta^{g-i+k-1}t_{k}t_{i-k+2})}{1+\zeta^{-i}}
  & (i=2,3,\ldots,g-1).
 \end{array}
\right.
\]
\begin{thm}[Schindler \cite{0801.14008}]
The $(i,j)$-th entry of the period matrix $\tau^{S}_g$ is obtained by
\begin{center}
$s_{i,j}=1-\dfrac{1}{t_1}\displaystyle\sum_{k=1}^{i}t_kt_{j-i+k}$
\end{center}
for $1\leq i\leq j\leq g$ and $s_{j,i}$ for $g\geq i>j\geq 1$.
\end{thm}
If we set $z_1=1/z$ and $w_1=\sqrt{-1}\, w/z^{g+1}$,
we obtain the biholomorphism from $C^{\prime}_{q,1}$ to $C_{q,1}$. 
This implies that a symplectic basis for Schindler's period matrix \cite{0801.14008}
is given by
\[
(A_g,A_{g-1},\ldots,A_1,B_g,B_{g-1},\ldots,B_1)
=(A_1,A_2,\ldots,A_g,B_1,B_2,\ldots,B_g)
\left(
 \begin{array}{cc}
  L_g & O \\
  O & L_g 
 \end{array}
\right),
\]
using the symplectic basis of natural type.
Here the $(i,j)$-th entry of the $g\times g$ matrix $L_g$
is $1$ for $i+j=g+1$ and $0$ otherwise.
It immediately follows that $L_{g}^{-1}=L_g$ and 
$\left(
 \begin{array}{cc}
  L_g & O \\
  O & L_g 
 \end{array}
\right)\in \mathrm{Sp}(2g,\Z)$.
From the equation
\[
\tau_g^{S}=(L_g+\tau_g\cdot O)^{-1}(O+\tau_g L_g),
\]
we have
\begin{prop}
The relation between the two period matrices $\tau_g$ and $\tau^{S}_g$ is 
\[
\tau^{S}_{g}=L_{g}\tau_{g}L_{g}.
\]
\end{prop}

\noindent
{\bf Acknowledgements.}
The author would like to thank Nariya Kawazumi and Takashi Taniguchi for their useful comments.
He also would like to thank the referee for valuable comments.
This work was partially supported by JSPS Grant-in-Aid for Young Scientists(B) 25800053 and
Fellowship for Research Abroad of Institute of National Colleges of Technology.
The work was performed while staying at the Danish National Research Foundation Centre of Excellence, QGM (Centre for Quantum Geometry of Moduli Spaces) in Aarhus University.
He is very grateful for the warm hospitality of QGM.

\providecommand{\bysame}{\leavevmode\hbox to3em{\hrulefill}\thinspace}
\providecommand{\MR}{\relax\ifhmode\unskip\space\fi MR }
\providecommand{\MRhref}[2]{%
  \href{http://www.ams.org/mathscinet-getitem?mr=#1}{#2}
}
\providecommand{\href}[2]{#2}

\end{document}